\documentclass[12pt]{extarticle}
\usepackage{amsmath, amsthm, amssymb, color}
\usepackage[colorlinks=true,linkcolor=blue,urlcolor=blue,citecolor=teal]{hyperref}
\usepackage{graphicx}
\usepackage{caption}
\usepackage{mathtools}
\usepackage{enumitem}
\usepackage{verbatim}
\usepackage{tikz}
\usepackage[all]{xy}
\usepackage[english]{babel}
\usepackage{amsmath}
\usepackage{amsthm}
\usepackage{amssymb}
\usepackage{cleveref}  
\usepackage{graphicx}
\usepackage{xcolor}
\usepackage{parskip}
\usepackage{wasysym}
\usepackage{tikz-cd}
\usepackage{extpfeil}
\usepackage{ytableau}

\usetikzlibrary{matrix}
\usetikzlibrary{arrows}
\usepackage{algorithm}
\usepackage{ytableau}
\usepackage{shuffle}
\usepackage[noend]{algpseudocode}
\usepackage{caption}
\usepackage[normalem]{ulem}
\usepackage{subcaption}
\oddsidemargin \evensidemargin
\usepackage{makecell}
\usepackage{multirow,array}

\newtheorem{theorem}{Theorem}
\newtheorem{proposition}[theorem]{Proposition}
\newtheorem{lemma}[theorem]{Lemma}
\newtheorem{corollary}[theorem]{Corollary}

\theoremstyle{definition}
\newtheorem{definition}[theorem]{Definition}
\newtheorem{problem}[theorem]{Problem}
\newtheorem{remark}[theorem]{Remark}

\newtheorem{example}[theorem]{Example}

\newtheorem{question}[theorem]{Question}

\numberwithin{theorem}{section}

\newcommand{\PP}{\mathbb{P}}
\newcommand{\RR}{\mathbb{R}}

\newcommand{\CC}{\mathbb{C}}

\newcommand{\NN}{\mathbb{N}}

\newcommand{\TT}{\mathbb{T}}
\newcommand{\T}{\mathcal{T}}

\newcommand{\C}{\mathbb{C}}
\newcommand{\N}{\mathbb{N}}
\newcommand{\R}{\mathbb{R}}

\newcommand{\p}{\mathbb{P}}

\newcommand{\G}{\mathcal{G}}
\newcommand{\U}{\mathcal{U}}

\newcommand{\ot}{\otimes}
\newcommand{\one}{{\bf 1}}

\newcommand{\SSS}{\mathfrak{S}}
\newcommand{\s}{\mathop{\mathbb{S}}\nolimits}
\newcommand{\id}{\mathrm{id}}

\newcommand{\GL}{\mathop{\rm GL}\nolimits}
\newcommand{\SL}{\mathop{\rm SL}\nolimits}
\newcommand{\Sym}{\mathop{\rm Sym}\nolimits}
\newcommand{\seg}{\mathop{\rm Seg}\nolimits}
\newcommand{\rk}{\mathop{\rm rk}\nolimits}

\newcommand{\Lie}{\mathop{\rm Lie}\nolimits}

\newcommand{\compositions}{P}
\newcommand{\rd}[1]{\left\lfloor #1\right\rfloor}

\newcommand{\InvRing}{\mathcal{I}}

\newcommand{\UnivMap}{\varphi}

\DeclareMathOperator{\End}{End}
\DeclareMathOperator{\sgn}{sgn}

\DeclareMathOperator{\Rk}{\mathrm{rk}}

\newcommand\restr[2]{{
		\left.\kern-\nulldelimiterspace 
		#1 
		\vphantom{\big|} 
		\right|_{#2} 
}}

\title{Decomposing tensor spaces via path signatures}
\author{Carlos Am\'endola, Francesco Galuppi, Ángel David Ríos Ortiz,\\ Pierpaola Santarsiero, Tim Seynnaeve}

\newcommand\blfootnote[1]{
  \begingroup
  \renewcommand\thefootnote{}\footnote{#1}
  \addtocounter{footnote}{-1}
  \endgroup
}

\date{}

\begin{document}
\maketitle
\begin{abstract}
The signature of a path is a sequence of tensors whose entries are iterated integrals, playing a key role in stochastic analysis and applications. The set of all signature tensors at a particular level gives rise to the universal signature variety. We show that the parametrization of this variety induces a natural decomposition of the tensor space via representation theory, and connect this to the study of path invariants. We also reveal certain
constraints that apply to the rank and symmetry of a signature tensor. 
\end{abstract}

	\section{Introduction}
 \blfootnote{\textit{2020 Mathematics Subject Classification}: 60L10, 60L70, 05E10, 15A72, 14N07}
A \textit{path} is a continuous map $X:[0,1]\to\R^d$. This mathematical object can be used to interpret a wide range of situations. From a physical transformation to a meteorological model, from a medical experiment to the stock market, almost anything that involves parameters changing with time can be described by a path. As is common in applied mathematics, for explicit computations one approach is to associate discrete objects 
to continuous objects. In our case, we associate to a path $X$, tensors of format $d \times d \times \cdots \times d$.

Suppose that the components $X_1,\dots,X_d$ of $X$ are sufficiently smooth. For every positive integer $k$, the \emph{$k$-th level signature} of $X$ is an order $k$ tensor $\sigma^{(k)}(X)\in(\R^d)^{\otimes k}$, whose $(i_1,\ldots,i_k)$-th entry is
\[\int_{0}^{1}\int_{0}^{t_k}\dots\int_{0}^{t_3}\int_{0}^{t_2}\dot{X}_{i_1}(t_1)\cdots\dot{X}_{i_k}(t_k)dt_1\cdots dt_k.
\]
By convention, we define $\sigma^{(0)}(X)=1$. The sequence $\sigma(X)=(\sigma^{(k)}(X))_k$ is the \textit{signature} of $X$. Signatures were first defined in \cite{Chen}, and they enjoy many useful properties. For instance, the signature allows to uniquely recover a sufficiently smooth path up to a mild equivalence relation (see \cite[Theorem 4.1]{chenuniqueness}). 
One way to think about them is as a non-commutative analog of moments used to describe a probability distribution.

In the last decades the definition of signature was extended to paths that are not necessarily differentiable, and signatures rose to prominence in stochastic analysis \cite[Chapter 7]{friz2010multidimensional}. They also allow to establish fruitful links between stochastics and algebraic geometry. This conversation started in \cite{AFS19}, and since then mathematicians have studied signature tensors with tropical, numerical and combinatorial techniques \cite{tropic,numeric,combinat}. In this paper we would like to introduce two different perspectives on this fascinating family of tensors. We are going to look at signatures from the viewpoints of representation theory and tensor decompositions. 
We aim to uncover connections between these branches of Mathematics and the theory of path signatures, introducing these approaches in a language accessible for people with different backgrounds. 
Aside from the results we present, we hope this paper will be useful as a starting point for researchers who would like to work in these areas, and can promote 
collaborations among different communities.
The representation-theoretic approach appears already in \cite[Section 4]{AFS19}. It is the algebraic framework that captures \emph{invariants} - roughly speaking, features of the path that remain the same under linear transformations. {On the other hand, the tensor decomposition viewpoint  in the context of signatures is original, and studied here for the first time.}

The paper is organized as follows. Necessary preliminaries about tensor algebras and the definition of the universal signature variety are recalled in Section \ref{section: univ}.  In Section \ref{section: Thrall} we introduce the central objects that drive our study, namely Thrall modules, coming from the parametrization of the universal signature variety. 
Understanding the Thrall modules from the point of view of representation theory turns out to be equivalent to some classical questions in algebraic combinatorics. We spell out these connections in Sections \ref{sec:ThrallDecomp} and \ref{sec:idempinv}, which are primarily of expository nature. Section \ref{sec:ThrallDecomp} is about the decomposition of the Thrall modules, which is related to the so-called \emph{Thrall problem}\footnote{This is the inspiration behind the terminology \emph{Thrall modules.}}, while in Section \ref{sec:idempinv} we describe the connection more explicitly, using higher Lie idempotents.
In Section \ref{sec:Invariants}, we apply the representation-theoretic methods of the preceeding sections to study \emph{invariants} of path signatures. {It is worthwhile to note that signatures may give fitting answers to representation theoretic problems. For instance, in \Cref{eg:LieIdempotents3} they allow us to distinguish between isomorphic copies of the same isotypic components. In \Cref{our new invariant} they provide a canonical choice for the basis of a space of invariants.}
Finally, in \Cref{section: tensor rank} we discuss \emph{tensor rank}. We show how the rank of signature tensors is tightly related to their symmetry. This allows us to give a new characterization of paths that have only finitely many nonzero log-signature tensors. Supporting code that verifies and extends all the computations presented here is available at the MathRepo repository \begin{center} \url{https://mathrepo.mis.mpg.de/TensorSpacesViaSignatures}.
\end{center}
\vspace{-0.2cm}
\subsubsection*{Acknowledgements}
The authors thank Bernd Sturmfels for suggesting  the topic. This project started while the authors were attending the Thematic Research Programme Algebraic Geometry with Applications to Tensors and Secants (AGATES) for which they acknowledge University of Warsaw, Excellence Initiative – Research University and the Simons Foundation Award No. 663281 granted to the Institute of Mathematics of the Polish Academy of Sciences for the years 2021-2023.

Am\'endola acknowledges support from DFG CRC/TRR 388 ``Rough Analysis, Stochastic Dynamics and Related Fields'', Projects A04 and B01. Galuppi acknowledges support from the National Science Center, Poland, projects ``Complex contact manifolds and geometry of secants'', 2017/26/E/ST1/00231, and ``Tensor rank and its applications to signature tensors of paths'', 2023/51/D/ST1/02363. R\'ios Ortiz is supported by the European Research Council (ERC) under the European Union’s Horizon 2020 research and innovation programme ERC-2020-SyG-854361-HyperK. Santarsiero was partially supported by the Deutsche Forschungsgemeinschaft (DFG, German Research Foundation) -- Projektnummer 445466444 and by the European Union under NextGenerationEU. PRIN 2022, Prot. 2022E2Z4AK. Seynnaeve is supported by Research foundation -- Flanders (FWO) -- Grant Number 1219723N.

\section{The universal signature variety}\label{section: univ}
In this section we recall some preliminaries about tensor algebras from \cite{Reutenauer}, and give the definition of the universal variety. This is an \emph{algebraic variety}, i.e., the zero set of a system of polynomial equations. See \cite[Section 2.2]{AFS19} for a friendly introduction in this context. The \emph{universal variety}, as introduced in \cite[Section 4]{AFS19}, is the closure of the set of all tensors which arise as the signature of paths. We will give a self-contained, coordinate-free, algebraic definition of the universal variety. The relation to signature tensors of paths follows from the works of Chen and Chow \cite{chen1957integration,chow}, see \cite[Section 4.3]{AFS19}.

For the rest of the paper, we fix a $d$-dimensional $\CC$-vector space $V$. This space $V$ should be thought of as the complexification of the space $\RR^d$ where the paths we are interested in live. 
 To define the universal variety, we need to introduce the \emph{tensor algebra}.
	This is the algebra $\TT((V))=\prod_{k\in \NN}V^{\ot k}$ of formal tensor series over $V$, with multiplication given by the tensor product. Elements of $\TT((V))$ are infinite sums $\T=\T_{(0)}+\T_{(1)}+\T_{(2)}+\cdots$, where each term $\T_{(k)} \in V^{\ot k}$. %
 We set
 $$
 \TT_0((V)):=\{\T \in \TT((V)) \mid \T_{(0)}=0\} \quad \hbox{ and } \quad \TT_1((V)):=\{\T \in \TT((V)) \mid \T_{(0)}=1\}.
 $$
 
\begin{definition}
The \emph{free Lie algebra} $\Lie(V)$ is the Lie subalgebra of $\TT_0((V))$ generated by $V \subset \TT_0((V))$. In other words, it is the smallest vector subspace of $\TT_0((V))$ that contains $V$ and is closed under taking commutator bracket. Elements of $\Lie(V)$ are called \emph{Lie polynomials}.
 \end{definition}
 One verifies that $\Lie(V)$ is a graded vector space%
 , i.e.\ $\Lie(V)= \bigoplus_{k \in \NN} \Lie^k V$, where
	$\Lie^k V := \Lie(V) \cap V^{\ot k}$.
    The truncation  $\bigoplus_{i=0}^k{\Lie^i V}$ will be written $\Lie^{\leq k} (V)$.
	Finally, we will write $\Lie((V)):= \prod_{k \in \NN} \Lie^k V \subset \TT_0((V))$, whose elements are called \emph{Lie series}.

\begin{example}
By definition 
    $\Lie^1V = V$ and 
    \[
    \Lie^2V = \operatorname{span} \left\{[v,w]:=v\ot w - w \ot v \mid v,w \in V\right\} = \bigwedge \nolimits^2V \subset V^{\ot 2},
    \]
    the space of skew-symmetric matrices. Similarly, we have
    \begin{align} \label{eq:Lie3}
    \begin{split}
        \Lie^3V &= \operatorname{span} \left\{[u,[v,w]] \mid u,v,w \in V\right\} \\
        &=\operatorname{span} \left\{u \ot v\ot w - u \ot w \ot v - v \ot w \ot u + w \ot v \ot u \mid u,v,w \in V\right\} \subset V^{\ot 3}.
    \end{split}   
    \end{align}
\end{example}

The space $\Lie^kV$ admits a basis indexed by Lyndon words of length $k$ in the alphabet of $d$ letters.
A string of letters is called a \emph{Lyndon word} if it is lexicographically smaller than all of its rotations.         %
    By \cite[Corollary 4.14]{Reutenauer}, %
    the number $\mu_{k,d}$ of such Lyndon words is  
\begin{equation}\label{eq: dimension Lie^k}
    \mu_{k,d} = \dim \Lie^kV = \frac{1}{k} \sum_{t\mid k}\mu(t)d^{\frac{k}{t}},
\end{equation}
where %
 $\mu%
 $ is the \emph{M\"obius function}. For instance, there are $\mu_{1,2}+\mu_{2,2}+\mu_{3,2}=2+1+2 = 5$ Lyndon words of length at most 3 in the alphabet $\{1,2\}$, namely $\{1,2,12,112,122\}$.

\begin{definition}

We denote by %
$\G(V) \subseteq \TT_1((V))$ the image of $\Lie((V))$ under the \emph{exponential map}
    \begin{align*}
 \exp: \TT_0((V)) &\to \TT_1((V))\\
 \T &\mapsto \sum_{n=0}^{\infty}{\frac{\T^{\ot n}}{n!}}.\nonumber
 \end{align*}
\end{definition}
As a consequence of the celebrated \emph{Baker-Campbell-Hausdorff formula} %
\cite[Corollary 3.3]{Reutenauer}, $\G(V)$ is %
a group under the tensor product%
. %
We now arrive at our central definition.
\begin{definition}
Let $p_k$ denote the projection $\TT((V)) \to V^{\ot k}$%
. The \emph{universal signature variety}, or simply \emph{universal variety} $\U_k(V)$, is the image of $\Lie((V))$ under the composition 
	   \begin{align} \label{eq:univVarietyMap}
\begin{split}
	\UnivMap_k:=p_k\circ\exp:\Lie((V)) & \longrightarrow V^{\ot k} \\
    \T_{(1)} + \T_{(2)} + \cdots & \mapsto \sum%
    {\frac{1}{\ell!}\T_{(\alpha_1)} \ot \cdots \ot \T_{(\alpha_{\ell})}},
\end{split}
\end{align}
where the sum is over all tuples of positive integers
$(\alpha_1,\ldots,\alpha_{\ell})$ with $\alpha_1+\cdots +\alpha_{\ell} = k$. 
	\end{definition}
Note that on the right-hand side there are no $\T_{(i)}$ with $i>k$ appearing. Therefore we can replace the domain of \eqref{eq:univVarietyMap} by $\Lie^{\leq k} (V)$. By \cite[Corollary 4.11]{AFS19}, the universal variety $\U_k(V)$ is the Zariski closure of the set of all tensors in $V^{\ot k}$ that are the signature of a path in $\R^d\subset V$. %

By construction, the universal variety $\mathcal{U}_k(V)$ can also be seen as
	    \[
	    \U_{k}(V)=\G(V) \cap V^{\ot k}.
	    \]
The group $\G(V)$ has a nice combinatorial characterization in terms of the \emph{shuffle product}.
\begin{definition}\label{def: shuffle}
Let $\TT(V^*) = \bigoplus_{k \in \NN}(V^{\ot k})^*$ be the graded dual of the tensor algebra. If $\beta: V^{\ot k} \to \CC$ and $\gamma: V^{\ot \ell} \to \CC$ are linear maps, their shuffle product $\beta \shuffle \gamma$ is a linear function $V^{\ot k + \ell} \to \CC$ defined by
\begin{equation}\label{eq:shuffle}
(\beta \shuffle \gamma)(v_1 \ot \cdots \ot v_{k + \ell}) \coloneqq \sum%
{\beta(v_{i_1} \ot \cdots \ot v_{i_k}) \cdot \gamma(v_{j_1} \ot \cdots \ot v_{j_{\ell}})},
\end{equation}
where the sum runs over all indices %
$i_1<\ldots<i_k$ and $j_1<\ldots<j_\ell$ such that $\{i_1,\ldots,i_k\} \sqcup \{j_1,\ldots,j_\ell\}= \{1,\ldots,k+\ell\}$.
The space $\TT(V^*)$ equipped with the shuffle product is known as the \emph{shuffle algebra}.
\end{definition}

\begin{example}\label{example: 12 shuffle 34}
Let $\beta, \gamma \in \TT(V^*)$ be the coordinate functions $\beta(\mathcal{T})=\mathcal{T}_{12}$ and $\gamma(\mathcal{T})=\mathcal{T}_{34}$. Here $\mathcal{T}$ denotes an element of $\TT((V))$, and $\mathcal{T}_{ij}$ is the $ij$-th coordinate of the degree two part.
Abusing notation, we will write
\[
\mathcal{T}_{12 \shuffle 34}\coloneqq(\beta \shuffle \gamma)(\mathcal{T}) = 
\mathcal{T}_{1234}+\mathcal{T}_{1324}+\mathcal{T}_{1342}+\mathcal{T}_{3124}+\mathcal{T}_{3142}+\mathcal{T}_{3412}.
\]
\end{example}

We can now state the \emph{shuffle identity}, proven in \cite[Theorem 3.2(iii)]{Reutenauer}.

\begin{theorem}\label{shuffle identity}
    The free Lie group $\mathcal{G}(V) \subset \mathbb{T}_1((V))$ is given by
    \begin{align}\label{eq:shuffleIdentity}
    \begin{split}
 \mathcal{G}(V) &=\{ \mathcal{T} \in \mathbb{T}_1((V)) ~\vert~  (\beta \shuffle \gamma)(\mathcal{T})=\beta(\mathcal{T})\gamma(\mathcal{T}), \hbox{ for all } \beta, \gamma \in \TT(V^*) \} \\
 &= \{ \mathcal{T} \in \mathbb{T}_1((V)) ~\vert~  \mathcal{T}_{I \shuffle J}=\mathcal{T}_I\mathcal{T}_J \hbox{ for all words } I,J\}. 
     \end{split}
 \end{align}
\end{theorem}

\section{Thrall modules}\label{section: Thrall}
Now we will explain how the parametrization of the universal variety in \eqref{eq:univVarietyMap} gives a decomposition of the 
space $V^{\ot k}$. As a warm-up, let us consider the case $k=2$. Then \eqref{eq:univVarietyMap} becomes
\begin{center}
\begin{tabular}{cccc}
$\UnivMap_2:$ &$V\oplus\bigwedge^2V$  &$\to$ &$V^{\ot 2}$\\
&$(v,A)$  &$\mapsto$ &$\frac{1}{2}v \ot v + A$.
\end{tabular}
\end{center}
We can see this map as the sum of the two maps  $(v,A) \mapsto \frac{1}{2}{v \ot v}$ and $(v,A) \mapsto A$. The linear span of the image of the first map is the space of symmetric matrices in $V^{\ot 2}$, and the image of the second map is the space of skew-symmetric matrices. We want to generalize this observation to $k\ge 2$%
. Recall that a \emph{partition} $\lambda \vdash k$ of $k$ of length $\ell$ is a tuple $\lambda = (\lambda_1,\ldots,\lambda_\ell)$ such that $\lambda_1\geq \ldots \geq \lambda_\ell >0$ and $\lambda_1+\cdots+\lambda_\ell=k$. We decompose the map $\varphi_k: \Lie^{\leq k}(V) \rightarrow V^{\otimes k}$ as $\UnivMap_k=\sum_{\lambda \vdash k}{f_{\lambda}}$, with
\begin{align}\label{eq: map f_lambda}
    f_{\lambda}: \Lie^{\leq k}(V) &\to
     V^{\ot k}  \\
    \T_{(1)} + \cdots + \T_{(k)} &\mapsto \sum_{\alpha \in \compositions(\lambda)}\frac{1}{\ell!}{\T_{(\alpha_1)} \ot \cdots \ot \T_{(\alpha_{\ell})}},\nonumber
\end{align}
where $\compositions(\lambda)$ is the set of distinct permutations of $\lambda$. For a partition $\lambda$, denote by $a_i(\lambda)$ the number of times the integer $i$ occurs in $\lambda$.
The map in \eqref{eq: map f_lambda} factors as $f_\lambda = g_\lambda \circ \nu_\lambda$ as follows
\[
\begin{matrix} %
\Lie^{\leq k}(V) &\xrightarrow{\nu_\lambda}& \Sym^{a_1(\lambda)}(\Lie^1 V) \otimes \cdots \otimes \Sym^{a_{k}(\lambda)}(\Lie^k V) &\xrightarrow{g_{\lambda}} &V^{\ot k} \\
    \T_{(1)} + \cdots + \T_{(k)} &\mapsto& \T_{(1)}^{\ot a_1(\lambda)} \otimes \cdots \otimes \T_{(k)}^{\ot a_k(\lambda)} &\mapsto& \sum_{\alpha}\frac{1}{\ell!}{\T_{(\alpha_1)} \ot \cdots \ot T_{(\alpha_{\ell})}}.
\end{matrix}
\]
The intermediate space was first studied by Thrall \cite[Section 7]{Thrall}. It will play a central role in this paper.
\begin{definition} \label{def:ThrallModule}
For any partition $\lambda$ of $k$, we define the \emph{Thrall module}
\[
W_{\lambda}(V) := \Sym^{a_1(\lambda)}(\Lie^1 V) \otimes \cdots \otimes \Sym^{a_{k}(\lambda)}(\Lie^k V).%
\]
\end{definition}
For the trivial partitions of $k$, we get $W_{(1,\dots,1)}(V) = \Sym^k V$ and $W_{(k)}(V) = \Lie^kV$. The following theorem 
is a corollary of the famous \emph{Poincar{\'e}-Birkhoff-Witt (PBW) theorem}, and was first proven %
in \cite{witt1937treue}. It shows that the Thrall modules give a decomposition of the tensor space $V^{\ot k}$. We will call this the \emph{Thrall decomposition}.
\begin{theorem} \label{thm:Wiso}
    The maps $g_{\lambda}$ are injective, and their sum gives an isomorphism of vector spaces
    \begin{equation} \label{eq:mainDecomposition}
    \sum_{\lambda \vdash k} g_{\lambda}: \bigoplus_{\lambda \vdash k}{W_{\lambda}(V)} \xrightarrow{\cong} V^{\ot k}.
\end{equation}
\end{theorem}
 \begin{proof}[Proof sketch]
   The PBW theorem states that for any Lie algebra $\mathcal{L}$ there is a vector space isomorphism between the symmetric algebra $\Sym(\mathcal{L})$, and the so-called \emph{universal enveloping algebra} $U(\mathcal{L})$. Combining this with the fact that $U(\Lie V) = \TT(V) \coloneqq \bigoplus_k V^{\ot k} $ yields the isomorphism \eqref{eq:mainDecomposition}.  
 \end{proof}

    In light of \Cref{thm:Wiso}, we will identify $W_{\lambda}(V)$ with the subspace
    \[
    \operatorname{span}\left\{\sum_{\alpha \in \compositions(\lambda)}{\T_{(\alpha_1)} \ot \cdots \ot \T_{(\alpha_{\ell})}} \, \middle| \, \T_{(1)} + \cdots + \T_{(k)} \in \Lie^{\leq k}(V)\right\}\subset V^{\ot k}.
    \]

\begin{example}\label{example:k=3Wmodules}
Let $k=3$. By \eqref{eq:univVarietyMap}, the map $\varphi_3: \Lie^{\leq 3}(V) \rightarrow V^{\otimes 3}$ becomes
    \begin{equation}\label{eq:kEquals3new}
\begin{matrix}
\varphi_3:  &  V \oplus \bigwedge^2V \oplus \Lie^3 V & \to & V^{\ot 3} \\
& v+A+L & \mapsto & \frac{1}{6}v^{\ot 3} + \frac{1}{2}(A \ot v + v \ot A) + L,
\end{matrix}
\end{equation}
which is the sum of the three maps
\begin{itemize}
    \item $f_{(1,1,1)}(v+A+L) = \frac{1}{6}v^{\ot 3}$, 
\item $f_{(2,1)}(v+A+L) = \frac{1}{2}(A\ot v+v\ot A)$, 
     \item $f_{(3)}(v+A+L) = L$.
\end{itemize}

For the partition $(2,1)$ we have
\[
\begin{matrix} %
    \nu_{(2,1)}: & \Lie^{1}V \oplus \Lie^2V &\to& \Lie^1 V \otimes \Lie^2 V \\
    & v + A & \mapsto & v \ot A
\end{matrix}
\]
and
\[
\begin{matrix} %
    g_{(2,1)}: & \Lie^1 V \otimes \Lie^2 V &\hookrightarrow& V^{\ot 3} \\
    & v \ot A & \mapsto & \frac{1}{2}(v \ot A + A \ot v).
\end{matrix}
\]

In this case the Thrall modules are
\begin{align*}
 W_{(3)}(V) &= \Lie^3V, \\ W_{(2,1)}(V) &= \operatorname{span}\left\{ A \ot v + v \ot A \,  \middle| \, v \in V, A \in \bigwedge \nolimits^2V \right\},\\
  W_{(1,1,1)}(V) &= \operatorname{span}\{v^{\ot 3} \mid v \in V\} = \Sym^3(V).
\end{align*}
\Cref{thm:Wiso} says that every tensor in $V^{\ot 3}$ can be uniquely written as a tensor in $W_{(1,1,1)}(V)$ plus a tensor in $W_{(2,1)}(V)$ plus a tensor in $W_{(3)}(V)$. %
\end{example}

\begin{remark} \label{rmk:GaluppiCoords}
 The Thrall decomposition \eqref{eq:mainDecomposition} yields a choice of coordinates on $V^{\ot k}$. In these coordinates, the maps $\nu_\lambda$ are given by monomials, as illustrated in \cite[Section 4]{G19}. This turns out to be useful for the study of geometric properties of $\U_{k}(V)$ and related varieties, see \cite{CGM}. %
\end{remark}

 The shuffle algebra also admits a Thrall decomposition: \begin{equation} \label{eq:grading}
 \TT(V^*) \cong \bigoplus_{k \in \NN}\bigoplus_{\lambda \vdash k}W_{\lambda}(V^*).   
 \end{equation}
The space $W_{\lambda}(V^*)$ consists of the linear functions $\TT((V)) \to \CC$ that only depend on the summand $W_{\lambda}(V)$.
We now show that $\TT(V^*)$ is graded by partitions as an algebra. This will play an important role in \Cref{sec:Invariants}. Here $\lambda \cup \mu$ denotes the partition with $a_{i}(\lambda \cup \mu) = a_{i}(\lambda) + a_{i}(\mu)$ for all $i$; for instance $(3,2,1) \cup (2,2) = (3,2,2,2,1)$.  

\begin{theorem} \label{thm:ShuffleThrall}
    If $\beta \in W_{\lambda}(V^*) \subset \TT(V^*)$ and $\gamma \in W_{\mu}(V^*) \subset \TT(V^*)$, then the product $\beta \shuffle \gamma$ lies in $W_{\lambda \cup \mu}(V^*)$. %
\end{theorem}
\begin{proof}
For any $\mathcal{T} = \T_{(1)} + \T_{(2)} + \cdots \in \Lie((V))$, we have by the shuffle identity (\Cref{shuffle identity})
\begin{equation} \label{eq:shuffleProof}
    (\beta \shuffle \gamma)(\exp(\mathcal{T})) = \beta(\exp(\mathcal{T})) \cdot \gamma(\exp(\mathcal{T})).
\end{equation}
Now the assumption $\beta \in W_{\lambda}(V^*)$ means that $\beta(\exp(\mathcal{T}))=\beta(f_{\lambda}(\mathcal{T}))$, where $f_{\lambda}$ is the map defined in \eqref{eq: map f_lambda}.
From this we see that if we rescale some $\T_{(i)}$ by a factor $t$, $\beta(\exp(\mathcal{T}))$ will rescale by a factor $t^{a_i(\lambda)}$. In other words, $\beta \circ \exp$ is a homogeneous polynomial of degree $a_i(\lambda)$ in $\T_{(i)}$. Similarly, $\gamma \circ \exp$ is a homogeneous polynomial of degree $a_i(\mu)$ in $\T_{(i)}$. But then by \eqref{eq:shuffleProof}, $(\beta \shuffle \gamma) \circ \exp$ is a homogeneous polynomial of degree $a_i(\lambda) + a_i(\mu)$ in $\T_{(i)}$, which can only be true if $\beta \shuffle \gamma \in W_{\lambda \cup \mu}(V^*)$.
\end{proof}

\section{Decomposition of the Thrall modules} \label{sec:ThrallDecomp}
The purpose of this section is to study the Thrall modules as $\GL(V)$-representations. To do so, we start by recalling standard notions of representation theory and we refer to \cite{FHrep} for a more detailed exposition.

A \emph{representation} of a group $G$ (also called \emph{$G$-module}) is given by a vector space $U$ and a group homomorphism $\rho: G \to \mathrm{GL}(U)$. We will be exclusively concerned with polynomial representations of the groups $\mathrm{GL}(V)$ and $\mathrm{SL}(V)$, where $V$ is a $d$-dimensional $\C$-vector space. This means that the entries of $\rho(M)$ are polynomials in the entries of $M \in G$. %
In particular, the spaces $V^{\ot k}$, $\Sym^k(V)$, $\bigwedge^k V$, $\Lie^k V$, and $W_{\lambda}(V)$ are all $\mathrm{GL}(V)$-representations since such constructions are functorial. A representation is \emph{irreducible} if it cannot be written as a direct sum of representations in a nontrivial way. Moreover, every representation can be written as a direct sum of irreducible representations. %

A very well-known decomposition of the tensor space $V^{\ot k}$ into irreducible $\GL(V)$-modules is the \emph{Schur-Weyl decomposition}:
 \begin{equation} \label{eq:SchurWeylDecomp}
    V^{\ot k} =  \bigoplus_{\mu \vdash k} \s_\mu(V)^{\oplus m_{\mu}},
\end{equation}
where $\mathbb{S}_\mu (V)$ is the \emph{Schur module} determined by the partition $\mu$ of $k$ and $m_{\mu}$ denotes the multiplicity of each module. For example, the Schur-Weyl decomposition of $
V^{\otimes 3}$ is 
\begin{equation} \label{eq:SchurWeyl3}
    V^{\otimes 3}= \Sym^3 (V)\oplus \bigwedge \nolimits^3 V\oplus (\mathbb{S}_{(2,1)}(V))^{\oplus 2}.
\end{equation}

The three summands $\Sym^3 (V)$, $\bigwedge ^3 V$ and $(\mathbb{S}_{(2,1)}(V))^{\oplus 2}$ are the \emph{isotypic components} of $V^{\ot 3}$ and they are canonically defined. However, the decomposition of %
the last isotypic component into the sum of the two irreducible modules $\mathbb{S}_{(2,1)}(V)$ is not unique. 
   More generally, for every representation, we have a unique decomposition into isotypic components, where every isotypic component decomposes into a number of copies of the same irreducible module. %

Schur modules are the main actors of the Schur-Weyl decomposition and since it will be useful later, let us briefly present a way to define them. Recall that a partition $\mu=(\mu_1,\dots,\mu_\ell)$ can be visualized as a grid of boxes all justified to the left having $\ell$ rows corresponding to the pieces of the partition $\mu_1, \dots,\mu_\ell$ respectively. A \emph{Young tableau} associated to a partition $\mu$ of $k$ is a filling of the partition with integers in $\{1,\dots,k \}$ without repetitions. It is called \emph{standard} if the rows and the columns are strictly increasing.
For example, the standard Young tableaux associated to the partition $\mu=(2,1)$ are 
\begin{equation*}
\ytableausetup{boxsize=1em}
    \begin{ytableau}
    1 & 2 \\ 3
    \end{ytableau} 
    \quad \text{ and }
\quad \ytableausetup{boxsize=1em} %
    \begin{ytableau}
    1 & 3 \\ 2
    \end{ytableau} \text{ .}
\end{equation*}
 \begin{definition}
     Let $\mu \vdash k$ be a partition of length at most $\dim V$.  
Fix a Young tableau $\tau$ associated to $\mu$ and call $\mathcal{R}_\tau, \mathcal{C}_\tau$ the subgroups of the symmetric group $\mathfrak{S}_k$ fixing rows, respectively columns, of the tableau. The \emph{Young symmetrizer} is the map
\begin{align}\label{eq:YoungSymmetrizer}
	\begin{split}
		c_\tau \colon V^{\otimes k}& \longrightarrow V^{\otimes k}\\
		v_1\otimes \cdots \otimes v_k &\mapsto \sum_{s \in \mathcal{C}_\tau} \sum_{t\in \mathcal{R}_\tau} \mathrm{sgn}(s) v_{t(s(1))}\otimes \cdots \otimes v_{t(s(k))}.
	\end{split}
\end{align}
The \emph{Schur module} $\mathbb{S}_\mu (V) \coloneqq c_\tau(V^{\otimes k})$ is the image of the Young symmetrizer. Notice that the construction still makes sense even when $\mu$ has length greater than $\dim V$, but in that case it will return the zero module.
 \end{definition} 

The symmetrizer $c_\tau$ depends on how we fill our grid of boxes corresponding to the partition $\mu$, but different fillings give isomorphic $\GL(V)$-representations, justifying the notation $\mathbb{S}_\mu (V)$.
However, we would like to stress that different choices of $\tau$ give rise to different embeddings of $\mathbb{S}_\mu (V)$ into $V^{\ot k}$. %

For example, to the partition $(3)$ is associated the Young tableau $\tau = \begin{ytableau}
    1 & 2 & 3
    \end{ytableau}$. In this case $\mathcal{R}_\tau = \SSS_3$ and $\mathcal{C}_\tau = \{\id\}$, so \eqref{eq:YoungSymmetrizer} yields $\s_{(3)}(V) = \Sym^3(V)$. More generally, for any $k$ we have $\s_{(k)}(V) = \Sym^k(V)$ and similarly $\s_{(1^k)}(V) = \bigwedge^kV$.
    For $\tau = \begin{ytableau}
    1 & 2 \\ 3
    \end{ytableau}$ we have $\mathcal{R}_\tau = \{\id, (12)\}$ and $\mathcal{C}_\tau = \{\id,(13)\}$, hence
    \[
    c_{\tau}(V^{\ot 3}) = \operatorname{span}\{ v_1 \ot v_2 \ot v_3 +  v_2 \ot v_1 \ot v_3 - v_3 \ot v_2 \ot v_1 - v_3 \ot v_1 \ot v_2 \mid v_i \in V\}.
    \]
\begin{remark}
    The most common way of decomposing the isotypic components in \eqref{eq:SchurWeylDecomp} into irreducibles is given by 
\begin{equation} \label{eq:SchurWeylStandard}
V^{\ot k} \cong \bigoplus_{\lambda \vdash k}\bigoplus_{\tau}c_\tau(V^{\ot k}), 
\end{equation}
where the second sum is over all standard Young tableaux on $\lambda$. However, we will soon see that the Thrall modules give an alternative to this.
\end{remark}

We now have two different decompositions of the tensor space $V^{\ot k}$: the decomposition \eqref{eq:SchurWeylDecomp} into isotypic components, and the decomposition \eqref{eq:mainDecomposition} into Thrall modules. Comparing these two decompositions amounts to determining the intersections 
\begin{equation} \label{eq:intersections}
   \s_{\mu}(V)^{\bigoplus m_{\mu}} \cap W_{\lambda}(V)\subset V^{\ot k}
\end{equation}
for all partitions $\mu, \lambda \vdash k$. In fact, there are two related questions one can ask. %

\begin{question}
    What are the intersections \eqref{eq:intersections} \emph{as $\GL(V)$-representations}?
\end{question}
\begin{question} \label{question:idempotents}
    What are the intersections \eqref{eq:intersections} \emph{as subspaces of $V^{\ot k}$}?
\end{question}

The second question will be addressed in \Cref{sec:idempinv}. Concerning the first one, \eqref{eq:intersections} is a multiple of the irreducible module $\s_\mu(V)$ and we only need to determine the multiplicity.
\begin{definition} \label{def:ThrallCoefficients}
For $\lambda$ and $\mu$ partitions of $k$, the \emph{Thrall coefficients} $a^{\lambda}_{\mu} \in \NN$ are the coefficients appearing in the $\GL(V)$-module decomposition
    \begin{equation} \label{eq:ThrallCoefficients}
    W_{\lambda}(V) \cong \bigoplus_{\mu \vdash k}  \s_{\mu}(V)^{\oplus a^{\lambda}_{\mu}}.
    \end{equation}
\end{definition}
The coefficients $a^{\lambda}_{\mu}$ do not depend on the dimension of $V$, because %
the construction of $W_{\lambda}(V)$ is functorial in $V$. Comparing the two decompositions tells us that $\sum_\lambda a_\mu^\lambda=m_\mu$ for every $\mu\vdash k$.

Before providing general methods to decompose Thrall modules into irreducibles, let us examine some small cases. The matrix case $k=2$ is covered in Section~\ref{section: Thrall}, so now let $k=3$. 
\begin{example}\label{example: d=k=3}
We compare the Schur-Weyl decomposition \eqref{eq:SchurWeyl3} to 
the Thrall decomposition
\[V^{\otimes 3}=W_{(1,1,1)}(V)\oplus W_{(2,1)}(V)\oplus W_{(3)}(V).\]
As we saw in \Cref{example:k=3Wmodules}, $W_{(1,1,1)}(V)=\s_{(3)}(V)=\Sym^3(V)$.
Now we only have to match the remaining components
	\[(\s_{(2,1)}(V) )^{\oplus 2}\oplus\s_{(1,1,1)}(V) \cong  W_{(2,1)}(V)\oplus W_{(3)}(V)\]
	or, equivalently, we have to understand if either $W_{(2,1)}(V)$ or $W_{(3)}(V)$ are reducible. One approach is to compute the dimensions of the modules involved. On the one hand we have
	\[
	\dim \s_{(1,1,1)}(V) = \dim \bigwedge \nolimits^3V = \binom{d}{3} \text{ and } \dim \s_{(2,1)}(V) = \frac{d^3-d}{3}
	\]
(see, for instance, \cite[Exercise 6.4]{FHrep} for the second formula). On the other hand, by \eqref{eq: dimension Lie^k} we get
	\begin{align*}
		\dim W_{(3)}(V)=\mu_{3,d}=\frac{d^3-d}{3} \hbox{ and } 
		\dim W_{(2,1)}(V) = \mu_{2,d}\cdot \mu_{1,d} =  \frac{d^3-d^2}{2}=\frac{d^3-d}{3}+\binom{d}{3}.
	\end{align*}
	So we find
	$W_{(3)}(V)\cong\s_{(2,1)}(V)$ and $\s_{(2,1)}(V) \oplus\s_{(1,1,1)}(V) \cong  W_{(2,1)}(V)$. 
\end{example}

We have already seen that $W_{(1^k)}(V) = \Sym^k(V) =\s_{(k)}(V)$.
In particular $W_{(1^k)}(V)$ is an irreducible $\GL(V)$-module. 
In the next example we indicate how to decompose certain Thrall modules using standard facts from representation theory. 
\begin{example}\label{rem:partionsWithOnly1And2} If $\lambda = (2^a,1^b)$, then 
$W_{\lambda}(V) = \Sym^a(\bigwedge^2V) \ot \Sym^b(V)$. Finding the decomposition of this representation is an exercise in representation theory. The decomposition of $\Sym^a(\bigwedge^2V)$ is given by
\begin{equation} \label{eq:SmWedge2}
\Sym^a\left(\bigwedge \nolimits ^2 V\right) \cong \bigoplus_{\lambda \vdash a}{\s_{(\lambda_1,\lambda_1,\lambda_2,\lambda_2,\ldots)}(V)},
    \end{equation}
see for instance  \cite[Equation 6.7.11]{Lands}. The tensor product of a Schur module with $\Sym^b(V)$ can then be computed using \emph{Pieri's rule} \cite[Equation 6.8]{FHrep}. For instance for $\lambda = (2,1)$ we have $W_{\lambda}(V) = \bigwedge^2V \ot V$, which by Pieri's rule is isomorphic to $\s_{(2,1)}(V) \oplus\s_{(1,1,1)}(V)$. This is consistent with \Cref{example: d=k=3}. 
\end{example}

 For a more systematic way of computing the decomposition of $W_{\lambda}(V)$ we can use character theory. For a detailed introduction to character theory, we refer the reader to \cite[Lecture 6]{FHrep}. Roughly speaking, to each representation $U$ of $\GL(\CC^d)$ we can associate a symmetric polynomial in $d$ variables known as the \emph{character}. Moreover, decomposing a representation into irreducibles amounts to writing its character as a combination of Schur polynomials.

The character of the representation $W_\lambda(V)$ is a symmetric polynomial called the \emph{higher Lie character} or \emph{Gessel-Reutenauer symmetric polynomial}, see for instance \cite[Equation 2.1]{GR93}. It can be computed in a purely combinatorial way, so that the problem of determining the Thrall coefficients becomes a problem in algebraic combinatorics: they are the coefficients of the expansion of the Gessel-Reutenauer symmetric polynomials in the Schur basis. 
In \cite{TheMathRepoFile}, we explain how to compute the higher Lie characters and how to algorithmically compute the Thrall decomposition. Even for a partition of $k=30$, the computation of the decomposition \eqref{eq:ThrallCoefficients} takes less than a second.

\begin{remark}
For $\lambda \vdash k\leq 4$, the decomposition of $W_{\lambda}(V)$ is multiplicity-free, meaning that each $a^{\lambda}_\mu$ is either zero or one. However, this pattern does not persist: already for $k=5$ we have $a^{(4,1)}_{(3,1,1)}=2$, i.e., the Thrall module $W_{(4,1)}(V)$ contains two copies of the Schur module $\s_{(3,1,1)}(V)$.
\end{remark}

    An open problem in algebraic combinatorics known as \emph{Thrall's problem} %
     asks for a combinatorial interpretation for the numbers $a_{\mu}^{\lambda}$. In \cite[Theorem 3.1]{Schocker03}, there is a closed formula for $a_{\mu}^{\lambda}$, but it contains denominators and rational numbers and is therefore not considered a solution to Thrall's problem. However, some cases are known. For example, for the partition $\lambda = (2^a,1^b)$, we can decompose $W_{\lambda}(V)$ by \Cref{rem:partionsWithOnly1And2}, yielding a combinatorial formula for $a_{\mu}^{\lambda}$. Also for $\lambda=(k)$ there is a combinatorial interpretation due to Klyachko \cite{klyachko}. We refer the reader to \cite[Section 8.6.1]{Reutenauer} and \cite{thrallproblem} for a more detailed overview. %

 The special cases of Thrall's problem listed above consider a particular $\lambda$ and arbitrary $\mu$. It is natural to wonder if one can also find formulas for some particular $\mu$ and arbitrary $\lambda$. For instance, the case of paths in the plane corresponds to $V=\CC^2$. In this case, the Schur modules $\s_{\mu}(V)$ are only nonzero if $\mu$ has length at most two, leading us to ask the following.
    
\begin{question}
    Given $\mu=(\mu_1,\mu_2)$, 
    is there a combinatorial interpretation for the multiplicities $a_{(\mu_1,\mu_2)}^\lambda$?
\end{question}

\section{Idempotents}
\label{sec:idempinv}

In this section, we address how the Thrall modules $W_{\lambda}(V)$ sit inside $V^{\ot k}$, cf.\ \Cref{question:idempotents}. 
To make this more precise, we need to introduce idempotents and to recall some more notions of representation theory. %

Giving a decomposition of a vector space $U$ into linear subspaces $U_i$ is equivalent to giving a collection of operators $\pi_i \in \End(U)$ that are idempotent ($\pi_i^2=\pi_i$), are pairwise orthogonal $\pi_i\pi_j=0$ if $i \neq j$), and sum to the identity. Explicitly, $\pi_j$ is the unique linear map that is constant on $U_j$ and $0$ on $\bigoplus_{i\neq j}{U_i}$. 
For a decomposition of $V^{\ot k}$ into $\GL(V)$-subrepresentations, these projection operators lie in $\End_{\GL(V)}(V^{\ot k})$, i.e.\ they commute with the $\GL(V)$-action. But the algebra $\End_{\GL(V)}(V^{\ot k})$ admits a beautiful description known as Schur-Weyl duality (see e.g.\ \cite[Chapter 6]{FHrep}). Recall that the group algebra $\CC[G]$ is defined by having a basis consisting of the elements of $G$, and multiplication induced by the group operation in $G$. Schur-Weyl duality then states that $\End_{\GL(V)}(V^{\ot k}) \cong \CC[\SSS_k]$, the group algebra of the symmetric group where a permutation $\sigma \in \CC[\SSS_k]$ corresponds to an endomorphism of $V^{\ot k}$ permuting the $k$ factors. Thus, giving a decomposition of $V^{\ot k}$ into (not necessarily irreducible) $\GL(V)$-submodules is equivalent to giving a decomposition $\one = \sum_i{e_i} \in \CC[\SSS_k]$ into pairwise orthogonal idempotents.

\begin{example}
The Young symmetrizers \eqref{eq:YoungSymmetrizer} can be identified with
 \[
 c_\tau = \left(\sum_{t \in \mathcal{R}_\tau}t\right)\cdot\left(\sum_{s \in \mathcal{C}_\tau}\sgn(s)s\right)\in\C[\SSS_k],
 \] 
 where $\tau$ is a Young tableau on a partition $\mu$. The \emph{normalized} Young symmetrizers $\tilde{c}_\tau=\frac{c_\tau}{m_\mu}$ are idempotents. For $k \leq 4$ they are orthogonal and sum to 1, so they are precisely the idempotents describing the decomposition \eqref{eq:SchurWeylStandard}. For $k \geq 5$, one has to orthogonalize the $\tilde{c}_\tau$ to get the idempotents describing the decomposition \eqref{eq:SchurWeylStandard}.
\end{example}

In order to describe the $W_{\lambda}(V)$ as subspaces of $V^{\ot k}$, all we have to do is to compute the corresponding idempotents. 
These idempotents were already studied in \cite{GarsiaReutenauer89} and are known as \emph{higher Lie idempotents}, we will denote them by $E_\lambda$. In that paper, there is a formula to compute the higher Lie idempotents \cite[Theorem 3.2]{GarsiaReutenauer89}, which has been 
implemented in \texttt{SageMath}. In particular, given a partition \verb|lambda| of the integer \verb|k|, the following code returns the higher Lie idempotent $E_{\lambda}$. See also \cite{TheMathRepoFile}.
\begin{verbatim}
DescentAlgebra(QQ, k).I().idempotent(lambda).to_symmetric_group_algebra()
\end{verbatim}

Let us continue our detailed study of the three factors case. In \Cref{example: d=k=3} we proved that the Thrall module $W_{(3)}(V)$ is isomorphic to the Schur module $\s_{(2,1)} (V)$. We also recall that in the Schur-Weyl decomposition of $V^{\otimes 3}$, the module $\s_{(2,1)} (V)$ appears with multiplicity two. We see in the following example how higher Lie idempotents play a central role to help us recognize which copy of $\s_{(2,1)} (V)$ is isomorphic to $W_{(3)}(V)$.
\begin{example} \label{eg:LieIdempotents3}
If $k=3$, then the \texttt{SageMath} algorithm outputs %
   \begin{align*}
  &E_{(3)} = \frac{1}{3}{\id} - \frac{1}{6}{(12)} - \frac{1}{6}{(23)} - \frac{1}{6}{(123)} - \frac{1}{6}{(132)} + \frac{1}{3}{(13)},\\
  &E_{(1,1,1)} =\sum_{\sigma \in \SSS_3}{\sigma}
  \qquad \text{and} \qquad E_{(2,1)} = \frac{1}{2}\id  - \frac{1}{2}(13).
   \end{align*}
We can verify this by hand by computing that
    \[
    T \cdot E_{\lambda} = \begin{cases}
        T & \text{if } T \in W_{\lambda}(V),\\
        0 & \text{if } T \in W_{\lambda'}(V) \text{ with } \lambda' \neq \lambda.
    \end{cases}
    \]
    For instance if $T \in W_{(3)}(V) = \Lie^3V$ we can use \eqref{eq:Lie3} to see that indeed $T \cdot E_{(1,1,1)} = T \cdot E_{(2,1)}=0$, but $T \cdot E_{(3)}=T$.
    From \Cref{example: d=k=3} we know that $W_{(2,1)}(V)$ further decomposes as a sum of two irreducibles. This corresponds to writing the idempotent $E_{(2,1)}$ as a sum of the two idempotents 
    \begin{itemize}
        \item $E_{(2,1);1} := \frac{1}{6}{\id} - \frac{1}{6}{(12)} - \frac{1}{6}{(23)} + \frac{1}{6}{(123)} + \frac{1}{6}{(132)} - \frac{1}{6}{(13)}$,
        \item $E_{(2,1);2} := \frac{1}{3}{\id} + \frac{1}{6}{(12)} + \frac{1}{6}{(23)} - \frac{1}{6}{(123)} - \frac{1}{6}{(132)} - \frac{1}{3}{(13)}$,
    \end{itemize}
    leading to the decomposition
    \begin{align*}
        V^{\ot 3} &= W_{(1,1,1)}(V) \oplus \left(W_{(2,1);1}(V) \oplus W_{(2,1);2}(V)\right) \oplus W_{(3)}(V) \\
        &\cong \s_{(3)}(V) \oplus \left(\s_{(1,1,1)}(V) \oplus \s_{(2,1)}(V)\right) \oplus \s_{(2,1)}(V).
    \end{align*}
    The first two idempotents $E_{(1,1,1)}$ and $E_{(2,1);1}$ agree with the idempotents giving $\s_{(3)}(V)$ and $\s_{(1,1,1)}(V)$ respectively in the Schur-Weyl decomposition.
    This comes as no surprise since the subrepresentations $\s_{(3)}(V)$ and $\s_{(1,1,1)}(V)$ have multiplicity one. On the other hand, the final two idempotents are not equal to any idempotent associated to any Young tableau $\tau$. In \cite{TheMathRepoFile}, we verify that 
    $W_{(3)}(V) = c_\tau(V^{\ot 3})$ where $\tau$ is the standard tableau
    \ytableausetup{boxsize=1em}
    \begin{ytableau}
    1 & 3 \\ 2
    \end{ytableau} and $c_\tau$ is the Young symmetrizer introduced in \eqref{eq:YoungSymmetrizer}. If instead we define the Young symmetrizer by first considering the subgroup fixing columns of a tableaux and then fixing the rows, we get the map $\tilde{c}_\tau$, which is analogous to the one in \eqref{eq:YoungSymmetrizer}.
We verify that $W_{(2,1);2}(V) = \tilde{c}_\tau(V^{\ot 3})$,  where $\tau$ is the standard tableau
    \begin{ytableau}
    1 & 2 \\ 3
    \end{ytableau}.
\end{example}

For $k=4$, some of the components are not equal to $c_\tau(V^{\ot k})$ or $\tilde{c}_\tau(V^{\ot k})$ for any tableau $\tau$, which shows that it can be challenging to describe the Thrall modules in general. 

\section{Invariants of paths} \label{sec:Invariants}

Let us fix a group $G$ acting on $V$. A \emph{tensor invariant} is a linear function $\beta \in \TT(V^*)$ %
such that $\beta(g\cdot \mathcal{T}) = \beta(\mathcal{T})$ for all $g \in G$ and $\mathcal{T} \in \TT((V))$. When we apply such a $\beta$ to a signature tensor $\mathcal{T}$, it extracts a number that is geometrically meaningful, in the sense that it does not depend on transforming the path by an element of $G$. It follows from \Cref{def: shuffle} that the shuffle product of two invariants is again an invariant. Hence the vector space of invariants forms a subalgebra of the shuffle algebra, which we will denote by $\InvRing_G(V)$.
In \cite[Section 7]{DiehlReizenstein}, the question was posed to find generators and relations for this algebra for various subgroups $G$ of $\GL(V)$. As we will soon see, the invariant ring $\InvRing_{\SL(V)}(V)$ is not finitely generated. To get around this, in \Cref{rmk:invTruncate} we introduce a truncated version which is finitely generated, and initiate the study of generators and relations of these restricted invariant rings.
\begin{remark}
The relevance of this ring $\InvRing_G(V)$ in relation to signature tensors is explained by the shuffle identity \eqref{eq:shuffleIdentity}: %
if $\mathcal{T}$ is the  signature of a path, then the shuffle product of two invariants of that path is just the usual product. 
The shuffle identity also implies that we can identify $\InvRing_G(V)$ with a classical object in invariant theory, namely the ring of invariants
\[\Sym(\Lie V^*)^{G}.\] 
See also \Cref{rmk:invTruncate} below. %
\end{remark}

Any group $G$ acts as the identity on the degree $0$ part of $\TT(V^*)$, i.e.\ elements of $\TT_0(V^*)$ are invariants for any $G$. In the case $G=\GL(V)$ there are no other invariants, because the action of a scalar matrix $a\cdot I$ rescales the elements of $\TT_k(V)$ by a factor $a^k$. An 
interesting case is $G=\SL(V) \coloneqq \{A \in \GL(V) \mid \det(A) = 1\}$. %
We will abbreviate $\InvRing_{\SL(V)}(V)$ to $\InvRing(V)$ for ease of notation. 

\begin{example} \label{eg:Levy}
Let $\dim V=2$. The map $\beta:\TT((V)) \to \C$ 
given by
    \begin{equation} \label{eq:Levy}
        \beta(\T) = \frac{\T_{12}-\T_{21}}{2}
    \end{equation}
is an $\SL(V)$-invariant. This has a geometric interpretation: if $\T$ is the signature of a piecewise smooth path $X:[0,1] \to \RR^2$, then $\beta(\T)$ 
is the signed area enclosed by the curve and the straight line segment joining the points $X(0)$ and $X(1)$, see \cite[page 54]{lyons2007differential}. In the stochastic context of rough paths, $\beta(\T)$ is known as the \emph{Lévy area}, see \cite[page 39]{friz2020course}.
\end{example}

The representation theory of $\SL(V)$ is almost identical to that of $\GL(V)$: if $d=\dim V$, then a complete set of pairwise nonisomorphic $\SL(V)$-representations is given by the Schur modules $\s_{\mu}(V)$, where $\mu$ runs over the partitions with at most $d-1$ parts. If $\mu$ has $d$ parts, then as $\SL(V)$-representations $\s_{\mu}(V) \cong \s_{\mu'}(V)$, where $\mu'=(\mu_1-\mu_d,\ldots,\mu_{d-1}-\mu_d)$. In particular, $\SL(V) \subseteq \GL(V)$ acts as the identity on a Schur module $\s_{\mu}(V)$ if and only if
$\mu$ is of the form $(\ell^d)$ for some $\ell \in \NN$, 
i.e. if and only if the Young diagram $\mu$ is a rectangle with exactly $d$ rows. This means that 
\[
\InvRing(V) = \bigoplus_{\ell \in \NN}\s_{(\ell^d)}(V^*)^{m_{(\ell^d)}} \subset \bigoplus_{\ell \in \NN}(V^*)^{\ot d\ell} \subset \TT(V^*). 
\]

Since $\TT(V^*)$ is graded by partitions by Theorem \ref{thm:ShuffleThrall}, the subalgebra  $\InvRing(V)$ is also graded in this way:
\[
\InvRing(V) = \bigoplus_{k\in\N}\bigoplus_{\lambda\vdash k} U_{\lambda}(V^*),
\]
where $U_{\lambda}(V^*) \coloneqq \InvRing(V) \cap W_\lambda(V^*)$. If $\lambda \vdash d\ell$ for some $\ell \in \NN$ then the dimension of $U_{\lambda}(V^*)$ is given by the Thrall coefficient $a^\lambda_{(\ell^d)}$; otherwise $U_{\lambda}(V^*)=0$. 
Computing these Thrall coefficients can give us partial information about the generators and relations of $\InvRing(V)$. %
For instance, we can make the following observations:
\begin{enumerate}[label=(\alph*)]
    \item \label{it:LieInvar} If $\beta \in U_{(d \ell)}(V^*)$ is nonzero, then $\beta$ cannot be expressed in terms of invariants of lower degrees. This means that we need at least $a^{(d\ell)}_{(\ell^d)}$ generators in degree $d\ell$.
    \item \label{it:Product} If $\beta \in U_{\lambda}(V^*)$ and $\gamma \in U_{\mu}(V^*)$ are both nonzero, then their product $\beta\cdot \gamma$ is a nonzero element of $U_{\lambda \cup \mu}(V^*)$. %
\end{enumerate}

Elements of $U_{(d\ell)}(V^*) \subset W_{(d\ell)}(V^*) = \Lie^{d\ell}(V^*)$ are known as \emph{Lie invariants}, and have received considerable attention in the classical literature; see \cite[page 208]{Reutenauer} for an overview. In particular $U_{(d\ell)}(V^*)$ is always nonzero, except in the case $\ell=1$ and $d\geq 3$ and in the case $\ell=2$ and $d\in\{2,3\}$. Together with \ref{it:LieInvar} above, this implies that 
 the algebra $\InvRing(V)$ is not finitely generated. 

One way of obtaining a ring that is finitely generated is to truncate the free Lie algebra and we can perform this operation since since $\Lie^{\leq m}(V^*)$ is finite-dimensional.
\begin{definition}\label{rmk:invTruncate}
    We define  \[
\InvRing_{\leq m}(V) \coloneqq \Sym(\Lie^{\leq m}(V^*))^{\SL(V)} = \bigoplus_{k \in \NN} \bigoplus_{\substack{\lambda \vdash k \\ \lambda_1 \leq m}} U_{\lambda}(V^*) \subset \InvRing(V).
    \]
\end{definition}
We remark that $ \InvRing_{\leq m}(V)$ is finitely generated by Hilbert's finiteness theorem (see for instance \cite[Theorem 2.2.10]{DerksenKemper}). 
    The ring $\InvRing_{\leq m}(V)$ is related to the \emph{rough Veronese variety} $\mathcal{R}_{k,m}(V)$ introduced in \cite[Section 5.4]{AFS19} and studied in \cite{G19} and in \cite[Section 2]{CGM}, which is the image of the restriction of the map \eqref{eq:univVarietyMap} to $\Lie^{\leq m}(V)$.

\begin{example}\label{example: alternating signatures}
Let $k=d$. Up to scaling, in $(V^*)^{\ot k}$ there is a unique invariant, given by the Schur module $\s_{(1^k)}(V^*)$. In \cite[Section 4]{AmendolaLeeMeroni} this invariant is called \emph{alternating signature}. If $d=2e$ is even, one can use \eqref{eq:SmWedge2} to see that the alternating signature is contained in $W_{(2^{e})}(V^*)$. %
This implies in particular that for a Lie polynomial $\T \in \Lie(V)$, its alternating signature only depends on the degree 2 part $\T_{(2)} \in \Lie^2V$. 
    More explicitly, we have $W_{(2^{e})}(V^*) = \Sym^2(\bigwedge^{e}V^*)$. This space contains a unique invariant, known as the \emph{Pfaffian}. In coordinates it is given by
    \begin{equation*}\label{eq:alternatingEven}
        \sum_{\sigma \in \SSS_{2{e}}}{\sgn(\sigma)\prod_{i=1}^{{e}}{\mathcal{T}_{\sigma(2i-1)\sigma(2i)}}}.
    \end{equation*}
    Compare with \cite[Lemma 4.1]{AmendolaLeeMeroni}. In the case $e=1$ we recover the signed area from \Cref{eg:Levy}.
    In the case where $k=d=2{e}+1$ is odd, the alternating signature lies in $W_{(2^{e},1)}(V^*) = \Sym^2(\bigwedge^{e}V^*) \ot V^*$. For %
 a Lie polynomial $\T \in \Lie(V)$, it can be computed as 
 \begin{equation*}\label{eq:alternatingOdd}
        \sum_{\sigma \in \SSS_{2{e}+1}}{\sgn(\sigma) \mathcal{T}_{\sigma(1)}\prod_{i=1}^{{e}}{\mathcal{T}_{\sigma(2i)\sigma(2i+1)}}}.
    \end{equation*}
Compare with \cite[Lemma 4.3]{AmendolaLeeMeroni}. 
\end{example}

Next we compute generators and relations for $\InvRing_{\leq m}(V)$ in small cases.

\begin{proposition}
    The invariant ring $\mathcal{I}_{\leq 2}(V)$ is a univariate polynomial ring, generated by the alternating signature.
\end{proposition}
\begin{proof}
    Since $\Lie^{\leq 2}(V^*) \cong V^* \oplus \bigwedge^2 V^*$, we compute 
    \[
    W_{(2^{a},1^{b})} = \Sym^{a}(\bigwedge \nolimits^2 V^*) \ot \Sym^{b}(V^*).
    \]
    Using Pieri's rule and the decomposition \eqref{eq:SmWedge2} of \Cref{rem:partionsWithOnly1And2}, one can verify that the only such spaces containing an invariant (i.e.\ a copy of $\s_{(\ell^d)}(V^*)$ for some $\ell$) are:
    \begin{itemize}
        \item $W_{(2^{\ell e})}$, in the case $d=2e$ is even.
        \item $W_{(2^{\ell e},1^\ell)}$, in the case $d=2e+1$ is odd. %
    \end{itemize}
    In both cases the multiplicity of $\s_{(\ell^d)}(V^*)$ is equal to one. 
    These invariants need to be the $\ell$-th powers of the alternating signatures (Pfaffians) mentioned in \Cref{example: alternating signatures}. %
\end{proof}

\begin{proposition}
    If $\dim V=2$, then the  invariant ring $\mathcal{I}_{\leq 3}(V)$ is freely generated by the unique invariant in $W_{(2)}$, and the unique invariant in $W_{(3,1)}$. 
\end{proposition}
\begin{proof}
     As $\SL_2$-representations, we have the isomorphisms $\Lie^2(V^*) \cong \s_{(2)}(V^*) \cong \CC$ and $\Lie^3(V^*) \cong \s_{(2,1)}(V^*) \cong V^*$, hence
\[W_{(3^{a_3},2^{a_2},1^{a_1})}(V^*) = \Sym^{a_3}(V^*) \ot \Sym^{a_1}(V^*).
    \]
    By Pieri's rule, this space contains an invariant if and only $a_1=a_3$, and the multiplicity is again equal to one. Keeping in mind \ref{it:Product}, this is enough information to conclude.
\end{proof}

We can use higher Lie idempotents to compute invariants in coordinates. In \cite{TheMathRepoFile} we implemented a function \verb|path_invariants| in \verb|Sage| which computes for each partition $\lambda$ of $d\ell$ the space of invariants in $W_{\lambda}(V^*)$. It first computes a basis of the space of invariants, which is the isotypic component $\mathbb{S}_{(\ell^d)}(V^*)^{\oplus {m_{(\ell^d)}}} \subset (V^*)^{\otimes k}$ in the Schur-Weyl decomposition. Next, it computes for each $\lambda$ the higher Lie idempotent $E_{\lambda}$, and projects the aforementioned basis using these idempotents.

    \begin{example}\label{our new invariant}
        Let us take $d=2$ and $k=4$. The multiplicity of $\s_{(2,2)}(V^*)$ in $(V^*)^{\ot 4}$ is equal to 2. 
        A basis for the isotypic component $\s_{(2,2)}(V^*)^{\oplus 2}$ is given by 
        \begin{equation} \label{eq:IsotypicBasis}
            \{\T_{1212}-\T_{1221}-\T_{2112}+\T_{2121},\T_{1122}-\T_{1221}-\T_{2112}+\T_{2211}\}.
        \end{equation}
        We  compute the Thrall coefficients:
        \[
        a^{\lambda}_{(2,2)} =
        \begin{cases}
            1 & \text{for } \lambda \in \{ (2,2), (3,1)\} \\
            0 & \text{otherwise,}
        \end{cases}
        \]
        and the higher Lie idempotents
        {\scriptsize
        \begin{align*}
            E_{(2,2)} =& \frac{1}{8} \big(\id -(132)-(1432)-(12)+(1243)+(13)(24)+(243)+(23)+(123)-(1234)-(34)-(234) \\ & -(1423)-(13)-(143)+(14)+(142)+(1342)+(12)(34)+(134)-(124)-(1324)-(24)+(14)(23)\big) \\
            E_{(3,1)} =& \frac{1}{12} \big(4\id+(132)+(1432)+(12)-2(1243)-2(13)(24)-2(243)-2(23)-2(123)+(1234)+(34)+(234)\\ & +(1423)+(13)+(143)-2(14)-2(142)-2(1342)-2(12)(34)-2(134)+(124)+(1324)+(24)+4(14)(23) \big).
        \end{align*}}%
        If we act with $E_{(2,2)}$ on either of the basis vectors \eqref{eq:IsotypicBasis}, we obtain the (up to scaling) unique invariant $\beta_{(2,2)}$ in $W_{(2,2)}$:
        \begin{equation}\label{eq:oldinv}
        \beta_{(2,2)}(\T) = \frac{1}{4}(\T_{1122}-\T_{1221}-\T_{2112}+\T_{2211}).
        \end{equation}
        As expected, this is simply the shuffle product of \eqref{eq:Levy} with itself; if $\T$ is a signature tensor, then $\beta_{(2,2)}(\T)$ is the square of the Levy area.
        Using $E_{(2,2)}$, we compute the unique invariant in $W_{(3,1)}$: 
        \begin{equation}\label{eq:newinv}
         \beta_{(3,1)}(\T) = \frac{1}{3}(-2\T_{1122}+\T_{1212}+\T_{1221}+\T_{2112}+\T_{2121}-2\T_{2211}).
        \end{equation}
        The invariant $\beta_{(2,2)}$ appears already in \cite[Remark 18]{DiehlReizenstein}, together with a noncanonical second basis vector of the space of invariants. In contrast, the invariant \eqref{eq:newinv} is uniquely defined. It would be interesting to have a geometric interpretation.
    \end{example}

\section{The universal variety and tensor rank}\label{section: tensor rank}
 
Both from a theoretical and from an applied viewpoint, the main way to study a tensor is to look at its decompositions. A tensor $T\in V^{\ot k}$ is called \emph{elementary} if there exist $v_1,\ldots,v_k\in V$ such that $T=v_1\ot\cdots\ot v_k$. A \emph{decomposition} of $T$ is a way to write
 \begin{equation*}\label{eq: tensor decomposition}
T= T_1+\cdots +T_r
 \end{equation*}
 as a sum of elementary tensors. The \emph{rank} of $T$, denoted by $\rk(T)$, is the minimum length of a decomposition of $T$.
This notion of rank generalizes matrix rank. %
As illustrated in \cite[Chapter 1.3]{Lands}, tensor decompositions and tensor rank find direct applications in many areas of applied mathematics. It is therefore natural to wonder about the rank of signature tensors.

Tensor rank is trickier to handle than matrix rank - for instance, we do not have an efficient algorithm to compute the rank of a given tensor. In this section we take the first steps to characterize the rank of a signature tensor. When we deal with tensor rank, it is often convenient to work in the projective space $\p(V^{\ot k})$. With a little abuse of notation, we will still denote by $\U_k(V)\subseteq \p(V^{\ot k})$ the projectivization of the universal variety. Recall that the \emph{Segre variety}
\[\seg_k(V)=\{[T]\in\p( V^{\ot k})\mid \rk(T)=1\}\] is the locus of tensors of rank 1. It is isomorphic to the product of $k$ copies of $\p(V)$ via the \emph{Segre embedding}, see for instance \cite[Exercise 2.14]{hart}.

\begin{theorem}\label{thm: symmetric iff rank 1}
Let $V$ be a $\C$-vector space of dimension $d$ and let $T\in\U_{k}(V)$ be a nonzero signature tensor. Then $T$ has rank 1 if and only if $T$ is symmetric.
\end{theorem}
\begin{proof} %
First assume that $T$ is symmetric. Then
\[
T \in (\operatorname{Im} \varphi_k) \cap \Sym^k(V) = (\operatorname{Im} \varphi_k) \cap W_{(1,\ldots,1)}(V) = \operatorname{Im} f_{(1,\ldots,1)} = \left\{ v^{\ot k} \mid v \in V \right\},
\]
therefore it has rank one.
Conversely, assume now that $T$ has rank 1, that is there exist $v_1,\ldots,v_k\in V$ such that $T=v_1\otimes \cdots \otimes v_k$. Since $T\in\U_{k}(V)$, there is a tensor series $\T\in\G(V)$ such that $T$ is the $k$-th level of $\T$. In order to prove that $T$ is symmetric, we are going to apply the shuffle identity, see Theorem \ref{shuffle identity}.
 Let
\[
m=\max\{n\in\N\mid v_1,\dots,v_n\mbox{ are scalar multiples of } v_1\}.
\]
We want to prove that $m=k$. Assume by contradiction that $m<k$. Choose a basis $\{e_1,\dots,e_d\}$ of $V$ such that $v_1=e_1$ and $v_{m+1}=e_d$. After rescaling,
\[T=e_1^{\ot m}\otimes e_d\ot(v_{m+2}\ot\cdots\ot v_k).\]
In order to find a contradiction, we want to prove that $\T_{1^mdK}=0$ for every word $K$ of length $k-m-1$. We will show that recursively, by considering all such words in the lexicographic order. By construction, if $L$ is a length $k$ word, $\T_L\neq 0$ only if the first $m$ letters of $L$ are ones. On the one hand we have 
\begin{equation} \label{eq:T1}
\T_{1} \cdot \T_{1^{m-1}dK}= \T_{1\shuffle (1^{m-1}dK)}= m \T_{1^mdK},
\end{equation}
because for every other word $w$ in the shuffle product $1^m \shuffle dK$ the first $m$ letters will not be all ones and therefore $T_w=0$. On the other hand we have 
\[
\T_{d} \cdot \T_{1^mK}=\T_{d\shuffle (1^mK)}= \T_{1^mdK} + \sum_{J}{\T_{1^mJ}}
\]
where the sum is over words $J$ that are obtained by inserting the letter $d$ somewhere in $K$. %
Note that all such words $1^mJ$ are either equal to $1^mdK$ or lexicographically smaller. We can assume recursively that we have already proven that $\T_{1^mJ}=0$ when ${1^mJ}$ is smaller, so $\mathcal{T}_{d} \cdot \mathcal{T}_{1^mK}$ is a multiple of $\mathcal{T}_{1^mdK}$. In either case we find 
\begin{equation} \label{eq:Td}
\T_{d} \cdot \T_{1^mK}= \lambda \T_{1^mdK}.
\end{equation}
for some nonzero scalar $\lambda$. Now if we had $\T_{1^mdK} \neq 0$, then \eqref{eq:T1} would imply that $\T_{1^{m-1}dK}\neq 0$ and \eqref{eq:Td} would imply that $\T_d\neq 0$, hence
\[
0\neq \mathcal{T}_d \cdot \mathcal{T}_{1^{m-1}dK}=\mathcal{T}_{d\shuffle(1^{m-1}dK)}.
\]
But every summand of $d \shuffle 1^{m-1}dK$ does not start with $m$ ones, so %
$\mathcal{T}_{1^mdK} = 0$, a contradiction. \end{proof}

Geometrically, Theorem \ref{thm: symmetric iff rank 1} means that the set-theoretic intersection between the universal variety and the Segre variety is exactly the Veronese variety. The theorem also tells us that, while we can certainly decompose a signature tensor $T$ as a sum of rank 1 tensors, in general we cannot decompose $T$ as a sum of rank 1 signature tensors. Indeed, if $T$ has rank $r>1$ and $T=T_1+\cdots+T_r$ is a decomposition as a sum of rank 1 signature tensors, then $T_1,\dots,T_r$ are symmetric by Theorem \ref{thm: symmetric iff rank 1}. Therefore $T$ is symmetric as well and so $T$ has rank 1, a contradiction.

As highlighted in \cite[Proposition 6.3]{AFS19}, if the level-one signature is nonzero, then the $(k-1)$-th level signature is a function of the $k$-th level and the first level. %
This feature of signatures together with \Cref{thm: symmetric iff rank 1} tells us that all the previous level of the signature are symmetric, and hence of rank one, once we show symmetry in one level.

\begin{proposition}\label{prop:symm}
Let $\T=\T_{(1)}+\T_{(2)}+\dots\in\Lie((V))$ be a Lie series with $\T_{(1)}\neq 0$. Fix an integer $k\geq 2$. If $\varphi_k(\T)$ is symmetric, then $\T_{(i)}=0$ for every $i\in\{2,\ldots,k\}$ and $\varphi_i(\T)$ is symmetric for every $i\in\{1,\ldots,k\}$.
\end{proposition}

\begin{proof}
The first statement is a consequence of the Thrall decomposition, see \Cref{thm:Wiso}. 
Indeed, by assumption $\varphi_k(\T)\in \Sym^k(V)= W_{(1,\dots,1)}$. This implies that all summands of $\varphi_k(\T)$ given via the Thrall decomposition must be zero. The summand in $W_{(i,1^{k-i})}$ is equal to $\T_{(i)} \ot \T_{(1)}^{\ot k-i}$. Since $\T_{(1)} \neq 0$ by assumption, we conclude that $\T_{(i)}=0$ for every $i\in \{2,\ldots,k\}$. It then follows that $\varphi_i(\T) = \T_{(1)}^{\ot i}$ is symmetric for every $i\in \{1,\ldots,k\}$.
\end{proof}

Our result above suggests a promising connection between studying the rank of signature tensors and the complexity of paths, especially in light of the recent characterization of straight line segments in terms of finite support of the log-signature  \cite[Theorem 1.4]{friz2023rectifiable}. We make this relation explicit in the following corollary.

\begin{corollary}\label{cor:FLS}
    Let $\T=\T_{(1)}+\T_{(2)}+\dots\in\Lie((V))$ be the log-signature of a rectifiable path with $\T_{(1)}\neq 0$ and nonzero signature tensors $\varphi_k(\T) \in V^{\ot k}$. Then the following statements are equivalent to the path being a straight line segment (up to reparametrization and translation):
\begin{enumerate}
    \item[(a)] $\T_{(i)}=0$ for every $i > 1$.
    \item[(a')] There exists $k \in \NN$ such that $\T_{(i)}=0$ for every $i > k$.
    \item[(b)] $\varphi_i(\T)$ is symmetric for every $i > 1$.
    \item[(b')] There exists $k \in \NN$ such that $\varphi_i(\T)$ is symmetric for every $i > k$.
    \item[(c)] $\varphi_i(\T)$ has rank 1 for every $i > 1$.
    \item[(c')] There exists $k \in \NN$ such that $\varphi_i(\T)$ has rank one for every $i > k$.
\end{enumerate}   
\end{corollary}
\begin{proof}
It is clear that $(a),(b),(c)$ imply their prime versions. The first converse $(a')\Rightarrow(a)$ is the content of \cite[Theorem 1.4]{friz2023rectifiable}, and establishes the equivalence to the path being a line segment. Such a path has $\varphi_k(\T)=\T_{(1)}^{\otimes k}$ for every $k\in\N$, so that $(a)\Rightarrow(b)$ and $(a)\Rightarrow(c)$. Furthermore, by \Cref{prop:symm} we have that the implications $(b')\Rightarrow(b)$ and $(b')\Rightarrow(a)$ hold. Finally, using \Cref{thm: symmetric iff rank 1} we obtain that $(b)\Rightarrow(c)$ and $(c')\Rightarrow(b')$.
\end{proof}

\Cref{cor:FLS} can be seen as a starting point of the study of rank and symmetries of signatures. We refer to the recent preprint \cite{GS24} for generalizations.

Next we briefly discuss the rank of a general signature tensor. We say that ``the general element of a set $X$ satisfies a certain property" if the subset of elements satisfying that property is dense in $X$ with respect to the Zariski topology. It is natural to wonder whether signature tensors can have any rank. This question has a nice geometric interpretation. When $k\ge 3$, the locus of tensors of rank at most $r$ does not need to be closed. Its closure is called the \emph{$r$-secant variety} to the Segre variety and denoted by
\[\mathrm{Sec}_r(\seg_k(V))=\overline{\{[T]\in\p( V^{\ot k})\mid \rk(T)\le r\}}.\]
Recall that $\dim\mathrm{Sec}_r(\seg_k(V))\le \min\{rd-1,d^k-1\}$. Saying that the general tensor has rank at most $r$ is equivalent to saying that $\mathrm{Sec}_r(\seg_k(V))=\p(V^{\ot k})$. Saying that the general signature tensor has rank at most $r$ is equivalent to saying that $\U_k(V)\subseteq\mathrm{Sec}_r(\seg_k(V))$. We state this as an open problem for future research.

\begin{problem}\label{conj: general rank}
What is the rank of a general signature tensor? In particular, determine whether such signature tensors have generic rank.
\end{problem} 

From a practical viewpoint, we are wondering whether knowing the rank of a given tensor gives us information to determine whether it is a signature or not. %
We now solve Problem \ref{conj: general rank} for matrices. %
From \eqref{eq:univVarietyMap} we know that the universal variety of matrices $\U_{2}(V)$ is the closure of the image of
\begin{align*}
V \oplus \bigwedge \nolimits^2 V&\longrightarrow V^{\ot 2} \label{eq: param signature matrix}\\
	(x,A) & \mapsto xx^T + A.\nonumber
\end{align*}

\begin{lemma} \label{lem:rankOfSkewSymPlusRankOneSym}
    Let $A \in \bigwedge^2 V$ be a skew-symmetric matrix, and $x \in V$. Then
    \[
\Rk(A+xx^T) = \begin{cases}
\Rk(A) & \text{if } x \in \mathrm{rowspan}(A), \\
\Rk(A)+1 & \text{if } x \notin \mathrm{rowspan}(A).
\end{cases}
    \]
\end{lemma}
\begin{proof}
    Let us write $\mathrm{rowspan}(A)=W \subseteq V$. Then $A \in \bigwedge^2W \subseteq \bigwedge^2V$ with $\dim W = \Rk(A)$. If $x \notin W$, then clearly $\Rk(A+xx^T) = \Rk(A)+\Rk(xx^T)=\Rk(A)+1$. On the other hand, if $x \in W$, then we have $A+xx^T \in W^{\ot 2}$, hence $\Rk(A+xx^T) \leq \dim W = \Rk(A)$. To show that this inequality cannot be strict, we apply the \emph{matrix determinant lemma} to $A$, viewed as a full-rank matrix in $W^{\ot 2}$:
    $$
	\det(A+xx^T)=\det(A)+ x^T\mathrm{adj}(A)x. 
	$$
    Since $\mathrm{adj}(A)$ is skew-symmetric the second summand vanishes, so we obtain $\det(A+xx^T)=\det(A) \neq 0$, hence $\Rk(A+xx^T) = \Rk(A)$.
\end{proof}
\begin{corollary}
If $A \in \bigwedge^2 V$ is a general skew-symmetric matrix and $x \in V$ is a general vector, then $\Rk(A+xx^T)=d$. In particular, the general element of $\U_2(V)$ has generic rank.
\end{corollary}
\begin{proof}
If $n$ is even then $\Rk(A) = d$ and we are done by \Cref{lem:rankOfSkewSymPlusRankOneSym}. If $d$ is odd then $\Rk(A)=d-1$ but $x \notin \mathrm{rowspan}(A)$ so again by \Cref{lem:rankOfSkewSymPlusRankOneSym} we have $\Rk(A+xx^T)=d$. Now, since $\U_2(V)$ contains elements $A+x^{\ot 2}$ of general rank $d$, it is not contained in any smaller secant variety of the Segre variety. Hence the general element of $\U_2(V)$ has rank $d$.
\end{proof}

By \cite[Theorem 6.1]{AFS19}, the universal variety $\U_k(V)$ has dimension $\mu_{1,d}+\ldots+\mu_{k,d}$ - see \eqref{eq: dimension Lie^k} - so we can compare it to the dimension of $\mathrm{Sec}_r(\seg_k(V))$ to get a first lower bound on the rank. Indeed, if $\dim\U_k(V)>\dim\mathrm{Sec}_r(\seg_k(V))$, then $\U_k(V)\not\subseteq\mathrm{Sec}_r(\seg_k(V))$ and therefore the general signature tensor has rank at least $r$. In this way we can see that the general rank of a signature is at least
\[
\rd{\frac{d^k(d-1) -d(d^{k/2}-1)}{(d-1)k(kd-k+1)}}-1.
\]

We close this section by showing an example where 
choosing coordinates compatible with the Thrall decomposition (cfr.\ \Cref{rmk:GaluppiCoords}) 
can also be interesting in a more classical algebraic geometric setting, revealing some inner structure of tensor varieties.

\begin{example}%
\label{example: easy equation for hyperdeterminant}
Denote by $\mathrm{Seg}_3(\mathbb{C}^2)\subset \PP^7$ the Segre variety given by three copies of the projective line. Call $\tau(\mathrm{Seg}_3(\mathbb{C}^2))$ the tangential variety of $\mathrm{Seg}_3(\mathbb{C}^2)$. If we set coordinates $a_{i,j,k}$ on $\PP^7$, the equation of $\tau(\mathrm{Seg}_3(\mathbb{C}^2))$ is the hyperdeterminant
\begin{align*}
	\hbox{Hdet}(A)=\left( \begin{vmatrix}
		a_{0,0,0} & a_{0,0,1} \\
		a_{1,0,0} & a_{1,1,1}
	\end{vmatrix} +\begin{vmatrix}
		a_{0,1,0} & a_{0,0,1} \\
		a_{1,1,0} & a_{1,0,1}
	\end{vmatrix} \right)^2 -4 \begin{vmatrix}
		a_{0,0,0} & a_{0,0,1} \\
		a_{1,0,0} & a_{1,0,1}
	\end{vmatrix} \cdot \begin{vmatrix}
		a_{0,1,0} & a_{0,1,1} \\
		a_{1,1,0} & a_{1,1,1}
	\end{vmatrix},
\end{align*}
see for example \cite[Chapter 14, Proposition 1.7]{GKZ}.
The projective version of the universal variety is the image of the  map  $ 	\varphi \colon \PP(\Lie^{\leq 3}(\mathbb{C}^2)) \dashrightarrow \PP(\mathbb{C}^2\otimes \mathbb{C}^2\otimes \mathbb{C}^2)$. Let $s_1,s_2,t_{12},u_{112},u_{122}$ be coordinates of the weighted projective space $\mathbb{P}(\Lie^{\leq 3}(\mathbb{C}^2)) $. Up to a scalar multiple, the pullback of the hyperdeterminant with respect to the map $ 	\varphi \colon \PP(\Lie^{\leq 3}(\mathbb{C}^2)) \dashrightarrow (\mathbb{C}^2)^{\otimes 3}$ has the elegant expression
	$$
\varphi^*(	\hbox{Hdet}(A))=	\left(s_2\,u_{112}+s_1\,u_{122}\right)^{2}\left(3\,t_{12}^{2}-4\,s_2\,u_{112}-4\,s_1\,u_{122}\right)%
	$$
The code performing the above computation is available on \cite{TheMathRepoFile}.  
\end{example}

We expect that similar factorizations exist for the equations defining other classical tensor varieties. For instance, perhaps one can use the theory of Thrall modules to provide a nice description of the equations of certain secant varieties. We propose a systematic study of these as a future research direction.

{\small
\bibliographystyle{alpha}
\bibliography{References.bib}
}

\noindent{\bf Authors' addresses}
\smallskip
\small 

\noindent Carlos Am\'{e}ndola:
Institut f\"ur Mathematik, Technische Universit\"at Berlin, Straße des 17. Juni 136, 10623 Berlin, Germany.
\hfill {\tt amendola@math.tu-berlin.de}

\noindent Francesco Galuppi: Faculty of Mathematics, Informatics and Mechanics, University of Warsaw, Banacha 2, 02-097 Warsaw, Poland.
\hfill {\tt galuppi@mimuw.edu.pl}

\noindent \'Angel David R\'ios Ortiz:
Université Paris-Saclay, CNRS, Laboratoir de mathématiques d'Orsay,  B\^{a}t,  307, 91405 Orsay, France.
\hfill {\tt angel-david.rios-ortiz@universite-paris-saclay.fr}

\noindent Pierpaola Santarsiero:
Universit\`a di Bologna, Dipartimento di Matematica, Piazza di Porta S. Donato 5, 40126 Bologna, Italy.
\hfill {\tt pierpaola.santarsiero@unibo.it}

\noindent Tim Seynnaeve: Department of Computer Science, KU Leuven, Celestijnenlaan 200A, 3001 Leuven, Belgium.
\hfill {\tt tim.seynnaeve@kuleuven.be}

\end{document}